\documentclass[10pt]{amsart}
\usepackage{amsfonts}
\usepackage{amsmath,amssymb}
\usepackage{amsthm}
\usepackage{ascmac}
\usepackage{graphicx}
\usepackage{multirow}
\usepackage{amscd}
\usepackage{indentfirst,csquotes}

\begin{document}
\title{Equivariant crossing numbers for two-bridge knots}
\author{Jundai Nanasawa}
\date{}
\address{Department of Pure and Applied Mathematics, School of Fundamental Science and
Engineering, Waseda University, 3-4-1 Okubo, Shinjuku, Tokyo 169-8555, Japan}
\email{jun15@fuji.waseda.jp}

\newtheorem{dfn}{Definition}
\newtheorem{thm}{Theorem}
\newtheorem{prop}{Proposition}
\newtheorem{lem}{Lemma}
\newtheorem{qst}{Problem}
\newtheorem{ex}{Example}
\newtheorem{cor}{Corollary}
\newtheorem{conj}{Conjecture}
\renewcommand{\figurename}{Figure }
\renewcommand{\tablename}{Table }
\renewcommand{\refname}{References}
\renewcommand{\abstractname}{Abstract}

\let\thefootnote\relax

\begin{abstract}
Symmetries of knots have been studied extensively, and strongly invertible knots are one of them. Lamm defined the equivariant crossing number $c_t(K)$, the minimum crossing number among all symmetric diagrams for a strongly invertible knot $K$. In this paper, we define $c_2(K)$ for two-bridge knots by restricting diagrams to two types. This gives an upper bound for $c_t(K)$. We give an algorithm to determine $c_2(K)$ for any two-bridge knot. The results of calculation by a computer up to 14 crossings are shown. As a corollary, we show 20 examples of knots up to 10 crossings in Rolfsen's knot table whose symmetry can be improved without increasing the number of crossings.
\end{abstract}
\maketitle

\section{Introduction}
\label{sec:intro}
There are many kinds of symmetries in knot theory. For a smooth knot $K$ in $S^3$, we focus on symmetries that are realized by an orientation preserving periodic homeomorphism $h$ from the pair $(S^3,K)$ to itself. Smith \cite{smith} proved that possible forms of fixed point set ${\rm Fix}(h)$ are $\emptyset$, $S^0$, $S^1$, and $S^2$. Furthermore, Fox \cite{fox} showed ${\rm Fix}(h)\cap K$ is homeomorphic to either $\emptyset$, $S^0$, or $S^1$. In particular, in the case that ${\rm Fix}(h)$ is homeomorphic to $S^1$, ${\rm Fix}(h)$ can be regarded as an unknotted circle by the positive solution of the Smith conjecture \cite{morgan}. Therefore, for a nontrivial knot, ${\rm Fix}(h)\cap K$ is homeomorphic to $\emptyset$ or $S^0$. These two symmetries are said to be {\it (cyclically) periodic} and {\it strongly invertible}, respectively. In this paper, we focus on strongly invertible knots.\par
 A smooth knot $K$ in $S^3$ is said to be {\it strongly invertible} if there is an orientation preserving homeomorphism $h:(S^3,K)\to (S^3,K)$ which satisfies the following conditions:\\
\quad 1. the set $F$ of fixed points of $h$ is homeomorphic to $S^1$,\\
\quad 2. $F$ intersects $K$ in two points,\\
\quad 3. $h^2=id$ and $h\neq id$.\\
Lamm introduced a new invariant $c_t(K)$, the minimum number of crossings among all symmetric diagrams for a strongly invertible knot $K$ \cite{lamm2}. The invariant $c_t(K)$ is called the {\it equivariant crossing number} of $K$ in this paper. Then the above question refers to the difference between $c_t(K)$ and the usual minimum crossing number $c(K)$. For two-bridge knots, Lamm conjectured that $c_t(K)=c(K)$ if and only if $K$ is one of the two types in Figure \ref{fig:lamm} (Conjecture \ref{conj:lamm}). This conjecture is based on Proposition 3.6 in \cite{sakuma} (Proposition \ref{prop:sakuma} in this paper). Lamm's conjecture has not been proved yet, but the types of diagrams in the conjecture can be used to give an upper bound for $c_t(K)$.\par
 To grasp $c_t(K)$ for two-bridge knots, we define the {\it two-bridge equivariant crossing number} $c_2(K)$ by restricting diagrams to two types in Lamm's conjecture. We can use the continued fraction expansion of the slope $p/q$ to detect $c_2(K)$ for a two-bridge knot $K(p,q)$. The calculation of $c_2(K)$ gives an upper bound for $c_t(K)$.\par
 It is well known that any two-bridge knot can be written as a continued fraction expansion consisting of only even integers. We call this an {\it even expansion}. However, an even expansion is not efficient as an upper bound for $c_2(K)$. We define a {\it semi-even expansion}, an expansion corresponding to one of the two types in Lamm's conjecture. This expansion can be easily done by ignoring the parities of integers in odd positions of the even expansion. We prove that the semi-even expansion is superior to the even expansion as an upper bound for $c_2(K)$. This enables us to obtain a finite algorithm to determine $c_2(K)$. The results using a computer up to $c(K)=14$ are shown. If we focus on two-bridge knots $K$ with $c_2(K)=c(K)$, there is a symmetric diagram whose number of crossings is minimum. Even if the diagram is not symmetric at first, a symmetric diagram of the knot can be easily obtained from the result of the expansion. 20 such examples up to 10 crossings are found in Rolfsen's knot table \cite{rolfsen}.\par
This paper is organized as follows. In Section \ref{sec:even}, we explain the even expansion of a slope. This is used in Section \ref{sec:lamm} to explain Lamm's conjecture. In Section \ref{sec:semieven}, we focus on the semi-even expansion. In Section \ref{sec:results}, we give the results of the calculation for $c_2(K)$ and its corollary. The details of the algorithm to calculate $c_2(K)$ are given in Appendix.\par
\section{Even expansion}
\label{sec:even}
In this section, we describe the expansion of a rational number consisting of only even integers. We call this an {\it even expansion}. An ordinary expansion consisting of only positive integers is called a {\it positive expansion}. Note that a positive expansion is unique for each rational number. Throughout this paper, we regard two knots $K_1,K_2$ as the same knot if and only if there is a homeomorphism from $(S^3,K_1)$ to $(S^3,K_2)$. The homeomorphism does not need to preserve the orientation of $S^3$ because taking the mirror image of $K$ does not affect the crossing number.\par
Let $K(p,q)$ be a two-bridge knot and $[a_1,a_2,\ldots,a_n]$ be its continued fraction expansion (Figure \ref{fig:standard2b}). In other words, $p$ and $q$ are coprime and
\[\frac{p}{q}=a_1+\frac{1}{a_2+\frac{1}{\cdots+\frac{1}{a_n}}}.\]
\begin{figure}[htbp]
\begin{center}
\includegraphics[scale=0.6]{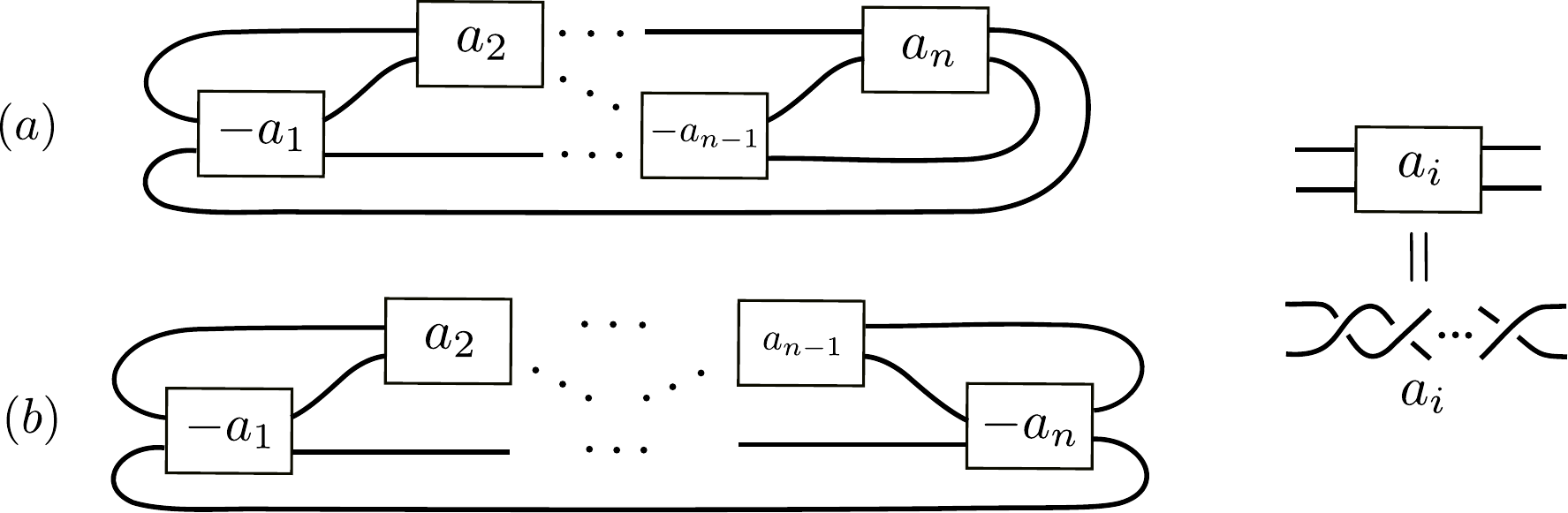}
\end{center}
\caption{Two-bridge knots. $(a)$ $n$ is even, $(b)$ $n$ is odd.\label{fig:standard2b}}
\end{figure}
Two-bridge knots are classified completely by the following theorem. See Theorem 2.1.3 in \cite{kawauchi1996} for example.
\begin{thm}
\label{thm:2b}
Two-bridge knots $K(p,q)$ and $K(p',q')$ represents the same knot if and only if $p=p'$ and $q\equiv q'^{\pm1}\mod p$.
\end{thm}
It is well known that $K(p,q)$ is a knot if and only if $p$ is odd. This fact and Theorem \ref{thm:2b} imply that if $K(p,q)$ is a knot, then $q$ can be chosen as an even integer such that $0<|q|<p$. 
\begin{prop}
\label{prop:even}
For any two-bridge knot $K(p,q)$ with odd $p$, even $q$, and $0<|q|<p$, the slope $p/q$ has a unique continued fraction expansion consisting of only even integers and the number of integers is also even. In other words, $p/q$ has an expansion with non-zero integers $a_i$ as follows:
\[\frac{p}{q}=[2a_1,2a_2,\ldots,2a_{2n}]
=2a_1+\frac{1}{2a_2+\frac{1}{\cdots+\frac{1}{2a_{2n}}}}.\]
\end{prop}
\begin{ex}
{\rm
For $K=6_3=K(13,5)$, we can replace the slope $13/5$ with $13/8$ using Theorem \ref{thm:2b}. We may assume that a rational number for each step of the expansion is positive by putting aside the negative sign if necessary. Divide the numerator by the denominator with a positive remainder. If the quotient is odd, add one to the quotient and subtract one from the remainder. By repeating this procedure, $13/8$ is expanded consisting of only even integers as follows:
\[\frac{13}{8}=2-\frac{3}{8}=2+\frac{1}{-2-\frac{2}{3}}=2+\frac{1}{-2+\frac{1}{-2+\frac{1}{2}}}=[2,-2,-2,2].\]
}
\end{ex}
\begin{dfn}
{\rm 
For a two-bridge knot $K(p,q)$, the continued fraction expansion in Proposition $\ref{prop:even}$ is called an {\it even expansion} of $p/q$.
}
\end{dfn}
\section{Lamm's conjecture}
\label{sec:lamm}
Two kinds of diagrams realize strong invertibility. Boyle uses the terms {\it intravergent} and {\it transvergent} \cite{boyle}. The diagram with an axis perpendicular to the plane of the diagram is said to be intravergent, and the diagram with an axis within the plane of the diagram is said to be transvergent (Figure \ref{fig:intra_trans}). The difference of isotopy is allowed in \cite{boyle}, but we define a {\it symmetric diagram} as follows.
\begin{figure}[htbp]
\begin{center}
\includegraphics[scale=0.6]{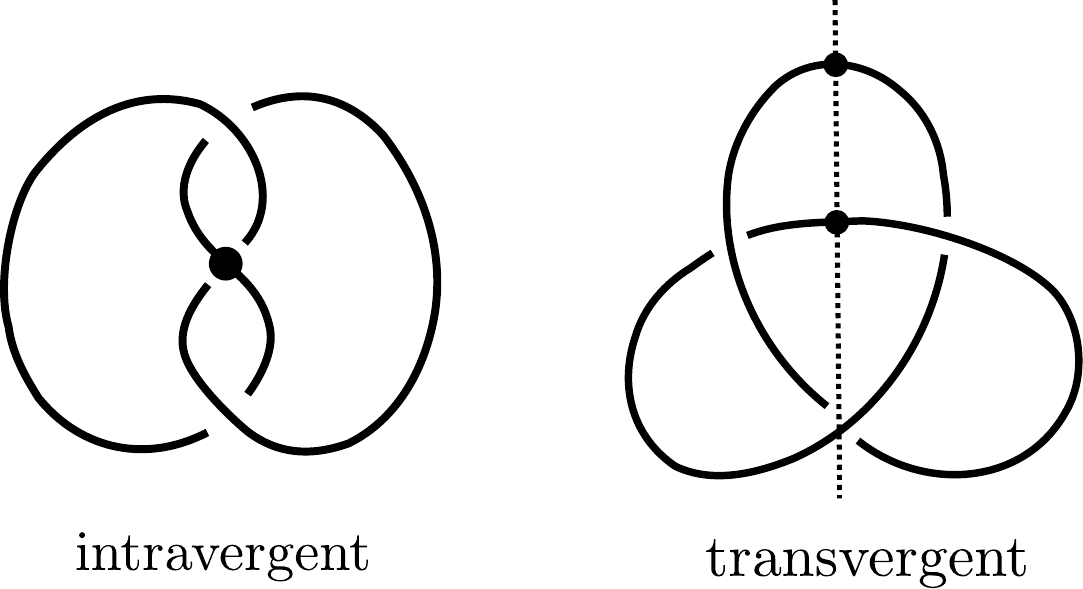}
\end{center}
\caption{The intravergent diagram and the transvergent diagram for the trefoil knot.\label{fig:intra_trans}}
\end{figure}
\begin{dfn}
{\rm 
A {\it symmetric diagram} for a strongly invertible knot $K$ is defined as its transvergent diagram, a diagram whose axis is a straight line on the plane of the diagram, and the diagram is unchanged under the half-rotation around the axis.
}
\end{dfn}
\begin{dfn}
{\rm 
For a strongly invertible knot $K$, the {\it equivariant crossing number} $c_t(K)$ is defined by the minimum crossing number among all symmetric diagrams for $K$.
}
\end{dfn}
Obviously $c(K)\leq c_t(K)$, and we are interested in the difference. To obtain $c_t(K)$ for alternating knots, we can restrict diagrams to investigate.
\begin{prop}
Let $K$ be an alternating prime strongly invertible knot and $D$ a diagram which realizes $c_t(K)$. Then $c_t(K)=c(K)$ if and only if $D$ is alternating.
\end{prop}
\begin{proof}
From Corollary $5.10$ in \cite{lickorish}, the number of crossings of a non-alternating diagram for an alternating knot is not minimum. Hence if $c_t(K)=c(K)$, then $D$ is alternating. Suppose $c_t(K)>c(K)$. If $D$ is alternating, $D$ is not reduced and the crossing in the non-reduced part can be removed preserving symmetry. Then the number of crossings of $D$ is less than $c_t(K)$. This contradicts the definition of $c_t(K)$. Therefore $D$ is not alternating.
\end{proof}
From this proposition, it is enough to consider alternating diagrams when we ensure $c_t(K)=c(K)$ for alternating knots. A primitive way is to allocate crossings to three regions, left, center (axis), and right, and investigate all patterns of the connections. However, this method is not effective for large crossing numbers. As far as the author knows,  $c_t(K)$ of the following knots are equal to $c(K)$:
\[6_1,6_2,7_1,7_2,7_3,7_4,7_5,7_6,7_7,8_1,8_2,8_3,8_4,8_5,8_6,8_{11},8_{12},8_{15},8_{16},8_{18},8_{19},8_{20},8_{21}.\]
In addition, $c_t(K)\leq c(K)+1$ holds for all knots $K$ with $c(K)\leq8$. These facts are based on the investigation of $c_2(K)$ in Section \ref{sec:results}, and Lamm's research \cite{lamm2} on three-bridge knots.\par
Lamm conjectured that for two-bridge knots, $c_t(K)=c(K)$ if and only if $K$ is one of the two types of diagrams.
\begin{conj}[Lamm \cite{lamm2}]
\label{conj:lamm}
Let $K$ be a two-bridge knot. Then $c_t(K)=c(K)$ if and only if $K$ can be represented by a continued fraction expansion $[a_1,\ldots,a_n]$ which is one of the following types (Figure \ref{fig:lamm}).\\
\quad$(A)$ All $a_i$ are positive, $n$ is even and $a_2,a_4,\ldots,a_n$ are all even.\\
\quad$(B)$ All $a_i$ are positive, $n$ is odd, integers are palindromic  such that $a_1=a_n,a_2=a_{n-1},\ldots$ and the central integer $a_{(n+1)/2}$ is odd.
\begin{figure}[htbp]
\begin{center}
\includegraphics[scale=0.5]{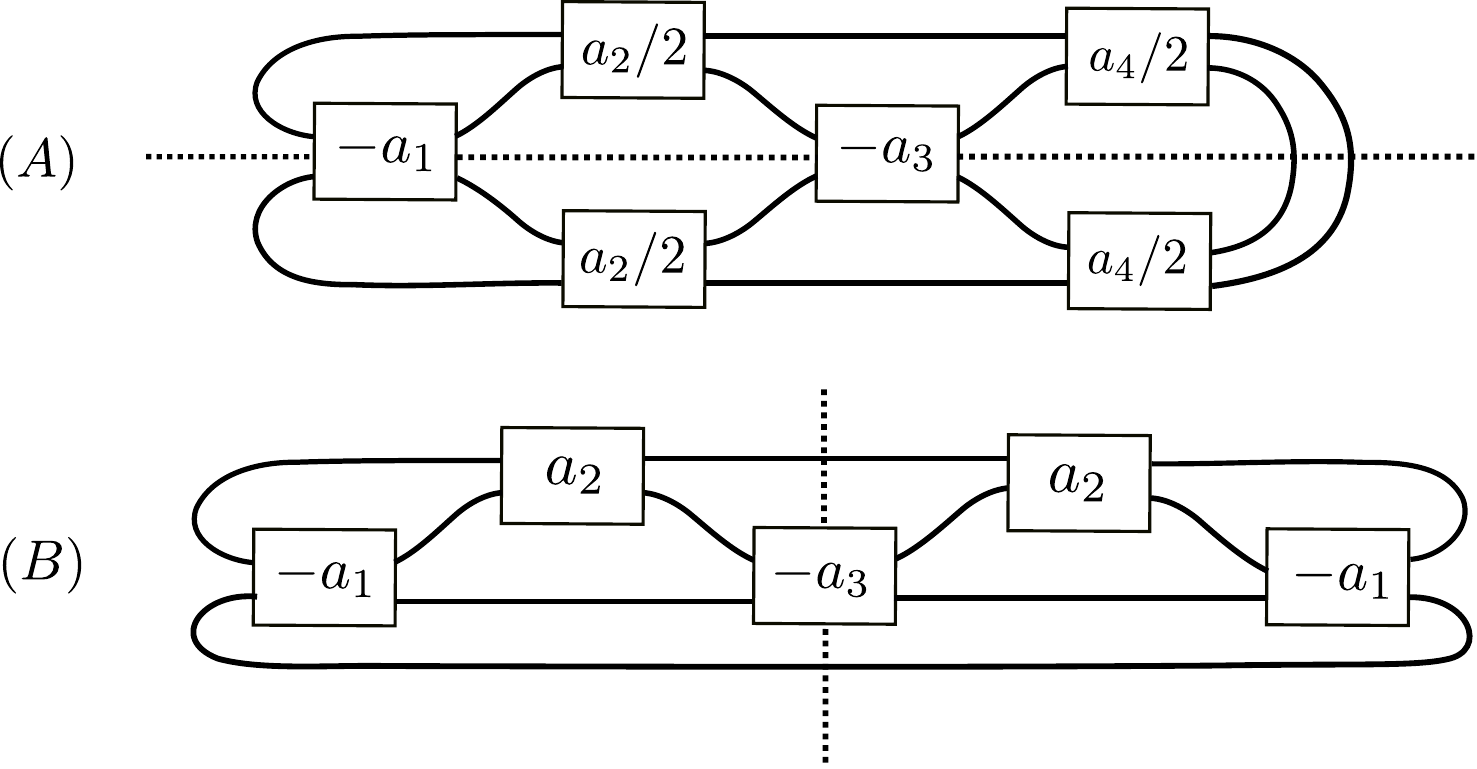}
\end{center}
\caption{Two types of Conjecture $\ref{conj:lamm}$.\label{fig:lamm}}
\end{figure}
\end{conj}
This conjecture comes from the following proposition by Sakuma (Proposition 3.6 in \cite{sakuma}). 
\begin{prop}[Sakuma \cite{sakuma}]
\label{prop:sakuma}
Let $K=K(p,q)$ be a two-bridge knot and $[a_1,\ldots,a_n]$ its even expansion with even $n$. Then $K$ has strong inversions as follows (Figure \ref{fig:sakumaprop}).\\
\quad$(1)$ If $q^2\not\equiv1\mod p$, then
\[I_1(a_1,a_3,\ldots,a_{n-1};a_2/2,a_4/2,\ldots,a_n/2),\]
\[I_1(-a_n,-a_{n-2},\ldots,-a_2;-a_{n-1}/2,-a_{n-3}/2,\ldots,-a_1/2).\]
\quad$(2)$ If $q^2\equiv1\mod p$, then
\[I_1(a_1,a_3,\ldots,a_{n-1};a_2/2,a_4/2,\ldots,a_n/2),\quad I_2(a_1,a_2,\ldots,a_{n/2}).\]
\begin{figure}[htbp]
\begin{center}
\includegraphics[scale=0.5]{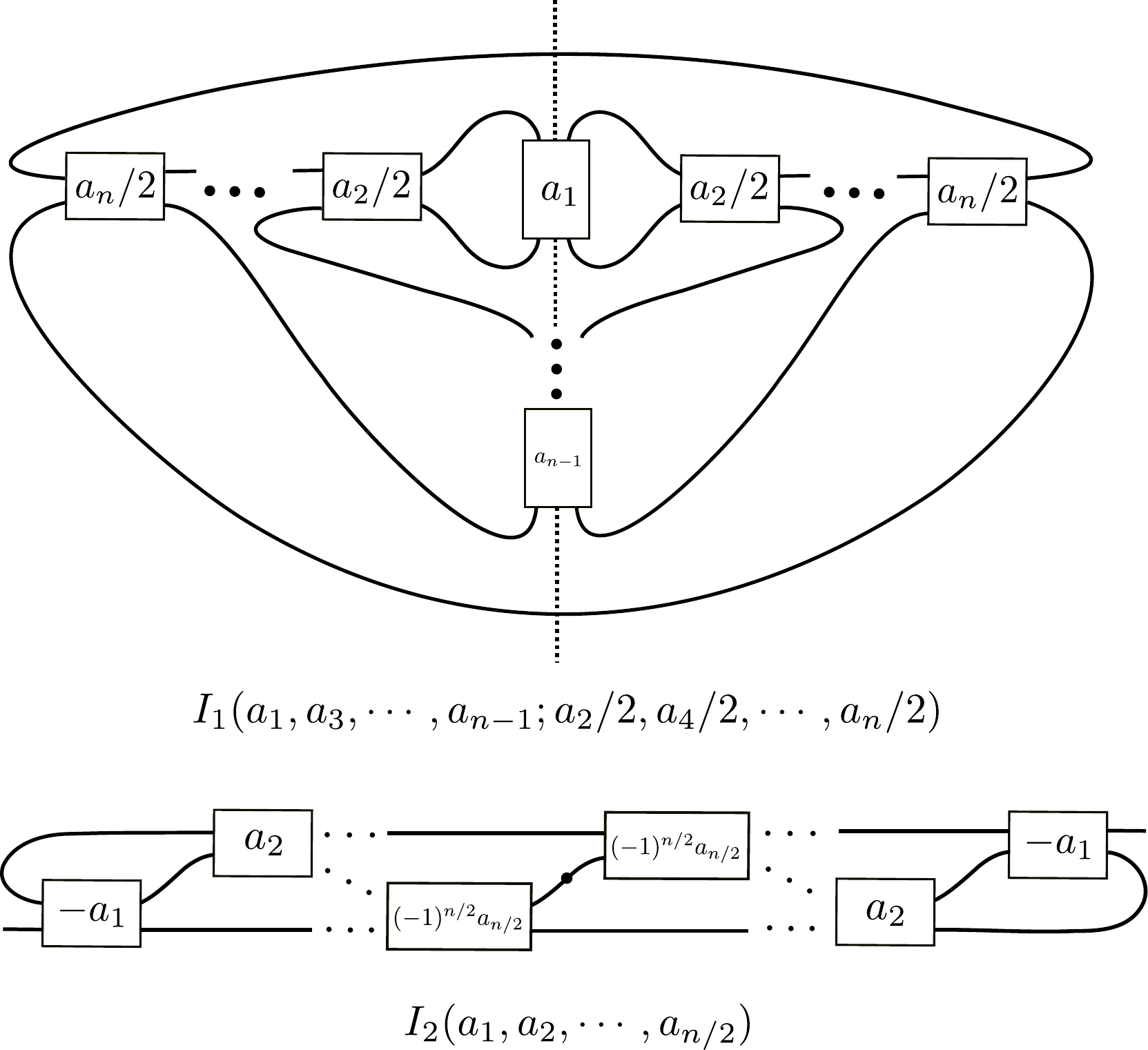}
\end{center}
\caption{Two strong inversions of two-bridge knots.\label{fig:sakumaprop}}
\end{figure}
\end{prop}
To evaluate $c_t(K)$ for two-bridge knots, we restrict diagrams to two types in Figure \ref{fig:lamm}.
\begin{dfn}
{\rm 
For a two-bridge knot $K$, the {\it two-bridge equivariant crossing number} $c_2(K)$ is defined by the minimum crossing number among all diagrams for $K$ of the type $(A)$ or $(B)$ in Figure \ref{fig:lamm}.
}
\end{dfn}
We focus on $c_2(K)$ of two-bridge knots from now on. The definition implies $c(K)\leq c_t(K)\leq c_2(K)$. This means $c_2(K)$ is an upper bound for $c_t(K)$. Investigating the continued fraction expansion of the slope is effective to determine $c_2(K)$.
\section{Semi-even expansion}
\label{sec:semieven}
In this section, the expansion corresponding to $(A)$ in Figure \ref{fig:lamm} is defined. In Conjecture \ref{conj:lamm}, the condition that all $a_i$ are positive corresponds to $c_2(K)=c(K)$. More generally, even in the case that some integers in the expansion are negative (i.e. $c_2(K)>c(K)$), we only need to consider two types of diagrams in Figure \ref{fig:lamm}. We define Type $A$ and Type $B$ as follows.\\
\quad(Type $A$) $n$ is even and $a_2,a_4,\ldots,a_n$ are all even.\\
\quad(Type $B$) $n$ is odd, integers are palindromic  such that $a_1=a_n,\ a_2=a_{n-1},\ldots$ and the central integer $a_{(n+1)/2}$ is odd.
\begin{dfn}
{\rm 
The {\it crossing number} of a continued fraction expansion $[a_1,\ldots,a_n]$ is defined by the value $\sum_{i=1}^n|a_i|$.
}
\end{dfn}
A simple way to know whether $c_2(K)=c(K)$ or not is to do the positive expansion of the slope $p/q$ $(p,q>0)$ and check that whether it is Type $A$ or $B$. Note that such expansion is unique. However, the differences of slopes in Theorem \ref{thm:2b} are allowed. A two-bridge knot $K(p,q)$($p$ is odd, $q$ is even, $0<q<p$, $\gcd(p,q)=1$) has four slopes such that $0<$ (denominator) $<$ (numerator):
\[\frac{p}{q},\quad\frac{p}{p-q},\quad\frac{p}{q'},\quad\frac{p}{p-q'}\quad(qq'\equiv1\mod p).\]
Slopes with even denominators are used to determine $c_2(K)$. On the other hand, if we know $c_2(K)\geq c(K)+1$, the following steps determine $c_2(K)$.\\
\\
\quad {\it Step 1.}\quad List all expansions of Type $A$ or $B$ with positive integers whose sum is $c(K)+1$.\\
\quad{\it Step 2.}\quad Add some negative signs for each expansion and calculate the continued fraction for all patterns. If one of the four slopes of $K$ appears, then $c_2(K)=c(K)+1$. Otherwise, go on to Step\ $3$.\\
\quad{\it Step 3.}\quad Add one more crossing and repeat from Step\ $1$.\\
\\
Since this procedure is complicated for large crossing numbers, we introduce another way of expansion. Type $A$ requires even integers in only even positions. Hence we can ignore the parities of integers in odd positions on Proposition $\ref{prop:even}$.
\begin{dfn}
{\rm 
Let $p/q$ be a rational number ($p$ is odd, $q$ is even, $0<q<p$ and $\gcd(p,q)=1$). A {\it semi-even expansion} is a continued fraction expansion consisting of even integers in all even positions.
}
\end{dfn}
\begin{prop}
For a semi-even expansion $p/q=[a_1,\ldots,a_n]$, $n$ is even if and only if $q$ is even.
\end{prop}
\begin{proof}
The expansion is as follows:
\[\frac{p}{q}=a_1+\frac{1}{\cdots+\frac{1}{a_{n-2}+\frac{1}{a_{n-1}+\frac{1}{a_n}}}}.\]
If $n$ is even, then $a_{n-1}+1/a_n$ is the form of $odd/even$. Taking the inverse and adding $a_{n-2}$ change the form into $even/odd$. This repeats alternatively and ends with $odd/even$. This implies $q$ is even. If $n$ is odd, then $a_{n-1}+1/a_n$ is the form of $odd/*$, where $*$ means either odd or even. Taking the inverse and adding $a_{n-2}$ change the form into $*/odd$. This repeats alternatively and ends with $*/odd$. This implies $q$ is odd.
\end{proof}
Since semi-even expansion generates a symmetric diagram of Type $A$, $c_2(K)$ is less than or equal to the crossing number of the expansion. Hence above three steps finish with finite times of repetition.
\begin{ex}
{\rm
Let $K$ be $6_3=K(13,8)$. Since $8\cdot5\equiv1\mod13$, four slopes of $K$ are
\[\frac{13}{8},\quad\frac{13}{5},\quad\frac{13}{5},\quad\frac{13}{8}.\]
The positive expansions are
\[\frac{13}{8}=1+\frac{5}{8}=1+\frac{1}{1+\frac{3}{5}}=1+\frac{1}{1+\frac{1}{1+\frac{2}{3}}}=1+\frac{1}{1+\frac{1}{1+\frac{1}{1+\frac{1}{2}}}}=[1,1,1,1,2],\]
\[\frac{13}{5}=2+\frac{3}{5}=2+\frac{1}{1+\frac{2}{3}}=2+\frac{1}{1+\frac{1}{1+\frac{1}{2}}}=[2,1,1,2].\]
Since these are neither Type $A$ nor $B$, we have $c_2(6_3)>6$. The slope $13/8$ is used for even and semi-even expansions. The even expansion is
\[\frac{13}{8}=2-\frac{3}{8}=2+\frac{1}{-2-\frac{2}{3}}=2+\frac{1}{-2+\frac{1}{-2+\frac{1}{2}}}=[2,-2,-2,2].\]
The crossing number is $8$. The semi-even expansion is
\[\frac{13}{8}=1+\frac{5}{8}=1+\frac{1}{2-\frac{2}{5}}=1+\frac{1}{2+\frac{1}{-2+\frac{1}{-2}}}=[1,2,-2,-2].\]
The crossing number is $7$ and this implies $c_2(6_3)\leq7$. Therefore $c_2(6_3)=7$.
}
\end{ex}
A semi-even expansion tends to have smaller crossing numbers than that of an even expansion. This can be proved as follows.
\begin{prop}
\label{prop:even_semieven}
For every slope $p/q$ ($p$ is odd, $q$ is even, $0<q<p$, $\gcd(p,q)=1$), the crossing number of the semi-even expansion is less than or equal to that of the even expansion.
\end{prop}
\begin{proof}
Although the even expansion and the semi-even expansion are defined for fractions of the form {\it odd}/{\it even}, we extend them to fractions of the form {\it even}/{\it odd}. In this case, integers in odd positions of the semi-even expansion have to be even. For example, $8/3=[2,1,2]$. Note that the number of integers in the even expansion of {\it even}/{\it odd} becomes odd.  Let $F$ be the set of all continued fractions and $s:F\to\mathbb{Z}$ a map which counts the crossing number as follows:
\[s([a_1,\ldots,a_n])=\sum_{i=1}^n|a_i|.\]
Let $A(p,q)=[a_1,\ldots,a_n]$ be the even expansion, $B(p,q)=[b_1,\ldots,b_m]$ a semi-even expansion. We prove $s(B(p,q))\leq s(A(p,q))$ by induction on the size of the numerator. This formula for fractions with small numerators is ensured by
\[\frac{3}{2}=[2,-2]=[1,2],\quad\frac{4}{3}=[2,-2,2]=[2,-1,-2].\]
Suppose the formula holds for fractions with a numerator less than $p$. It is enough to show for the case $A(p,q)\neq B(p,q)$. Compare two expansions and assume they vary in the $i$-th position for the first time. Note that $i$ is odd. From the procedures of expansions, the integer in the $i$-th position of $A(p,q)$ is greater than that of $B(p,q)$ by exactly one. Let $C(p,q)$ be a new expansion which is the semi-even expansion from the $(i+1)$-th position of $A(p,q)$. In other words, $C(p,q)$ is the even expansion until the $i$-th position and the semi-even expansion from the $(i+1)$-th position. For example, in the case of $13/8$,
\[A(13,8)=[2,-2,-2,2],\quad B(13,8)=[1,2,-2,-2],\quad C(13,8)=[2,-2,-1,-2].\]
Now, we prove the following formula:
\[s(B(p,q))\leq s(C(p,q))\leq s(A(p,q)).\]
We focus on the first inequality because the second inequality comes from the hypothesis of induction. Considering the procedures of expansions in the $i$-th position, two fractions to expand in the next step of $B(p,q)$ and $C(p,q)$ are of the form $P/Q$, $P/(P-Q)$, where $P,Q\in\mathbb{Z}$, $0<Q<P$, $\gcd(P,Q)=1$, $P$ is even, and $Q$ is odd. Note that both fractions are chosen to be positive as signatures do not affect the number of crossings. For example, in the case of $13/8$, values are $i=1$, $P=8$, and  $Q=3$. Since $s(B(p,q))$ is less than $s(C(p,q))$ by exactly one up to the $i$-th position, it is enough to prove
\[|s(B(P,Q))-s(B(P,P-Q))|\leq1.\]
Since we are considering the pair of $P/Q$ and $P/(P-Q)$, we may assume $Q<P/2<P-Q$. Note that $B(P,P-Q)$ always starts from 2.
\begin{itemize}
\item \underline{The case $Q>P/3$}\\
$B(P,Q)$ starts from $2$. Using some integers $\alpha,\beta_1,\ldots,\beta_k$,
\[\frac{P}{Q}
=2+\frac{1}{\frac{Q}{P-2Q}}
=[2,\alpha,\beta_1,\ldots,\beta_k].\]
Then
\[\frac{P}{P-Q}
=2-\frac{1}{\frac{P-Q}{P-2Q}}
=2-\frac{1}{\frac{Q}{P-2Q}+1}
=[2,-(\alpha+1),-\beta_1,\ldots,-\beta_k].\]
Hence $|s(B(P,Q))-s(B(P,P-Q))|\leq1$ holds.
\item \underline{The case $Q<P/3$}\\
$B(P,Q)$ starts from $2\alpha$ $(\alpha\geq2)$. Using some integers $\beta_1,\ldots,\beta_k$,
\[\frac{P}{Q}
=2\alpha+\frac{1}{\frac{Q}{P-2\alpha Q}}
=[2\alpha,\beta_1,\ldots,\beta_k].\]
Then
\begin{eqnarray*}
\frac{P}{P-Q}
&=&2-\frac{1}{\frac{P-Q}{P-2Q}}
=2-\frac{1}{1+\frac{1}{\frac{P-2Q}{Q}}}
=2-\frac{1}{1+\frac{1}{2\alpha-2+\frac{1}{\frac{Q}{P-2\alpha Q}}}}\\
&=&[2,-1,-(2\alpha-2),-\beta_1,\ldots,-\beta_k].
\end{eqnarray*}
Hence $|s(B(P,Q))-s(B(P,P-Q))|\leq1$ holds.
\end{itemize}
Therefore we have $s(B(p,q))\leq s(C(p,q))\leq s(A(p,q))$ and this implies $s(B(p,q))\leq s(A(p,q))$.
\end{proof}
\section{Results of calculation for $c_2(K)$}
\label{sec:results}
Proposition \ref{prop:even_semieven} implies that semi-even expansion is superior to even expansion as an upper bound for $c_2(K)$. However, semi-even expansion does not cover all cases of Type $A$. Hence $c_2(K)$ is not determined immediately when the crossing number of the semi-even expansion is greater than or equal to $c(K)+2$. In that case, we have to investigate all possible patterns of expansions. Using the results so far, the algorithm for calculating $c_2(K)$ is given. This appears in Appendix. The author programmed this procedure using Python and calculated $c_2(K)$ from the data in KnotInfo \cite{knotinfo} up to 12 crossings. Another way to obtain the data for slopes of two-bridge knots whose crossing numbers are $n$ is as follows. List all positive expansions with $n$ crossings, calculate the slopes, and remove overlaps of knot types by Theorem \ref{thm:2b}. Theoretically, that list makes it possible to calculate $c_2(K)$ for arbitrarily large $n$, but needs much time. Table \ref{tb:ct} shows data up to 14 crossings. The results of the two methods up to 12 crossings were the same.\par
\begin{table}[htb]
	\centering
	\begin{tabular}{|c|c||c|c|c|c|c|}  \hline
		\multirow{2}{*}{$c(K)$} & \multirow{2}{*}{All knots} & \multirow{2}{*}{Two-bridge knots} & \multicolumn{4}{c|}{$c_2(K)$} \\ \cline{4-7}
		 & & & $c(K)$ & $c(K)+1$ & $c(K)+2$ & $c(K)+3$ \\ \hline \hline
		3 & 1 & 1 & 1 & & & \\ \hline
		4 & 1 & 1 & 1 & & & \\ \hline
		5 & 2 & 2 & 2 & & & \\ \hline
		6 & 3 & 3 & 2 & 1 & & \\ \hline
		7 & 7 & 7 & 7 & & & \\ \hline
		8 & 21 & 12 & 7 & 5 & & \\ \hline
		9 & 49 & 24 & 17 & 7 & & \\ \hline
		10 & 165 & 45 & 17 & 25 & 3 & \\ \hline
		11 & 552 & 91 & 44 & 36 & 11 & \\ \hline
		12 & 2176 & 176 & 49 & 98 & 26 & 3 \\ \hline
		13 & 9988 & 352 & 109 & 152 & 89 & 2 \\ \hline
		14 & 46972 & 693 & 128 & 351 & 177 & 37 \\ \hline
	\end{tabular}
	\caption{The number of knots whose $c_2(K)$ is $c(K)+j$ $(j=0,1,2,3)$ up to $14$ crossings.}
	\label{tb:ct}
\end{table}
Concerning results so far, we can improve the symmetry of two-bridge knots in Rolfsen's knot table \cite{rolfsen}. For knots $K$ such that $c_2(K)=c(K)$ and not drawn symmetrically, we can deform it into a symmetric diagram preserving the crossing number. 20 such examples are found up to 10 crossings. Figure \ref{fig:rolfsen} shows one of them. All examples are shown in Figure \ref{fig:rolfsen_sym_table}.\par
\begin{figure}[htbp]
\begin{center}
\includegraphics[scale=0.6]{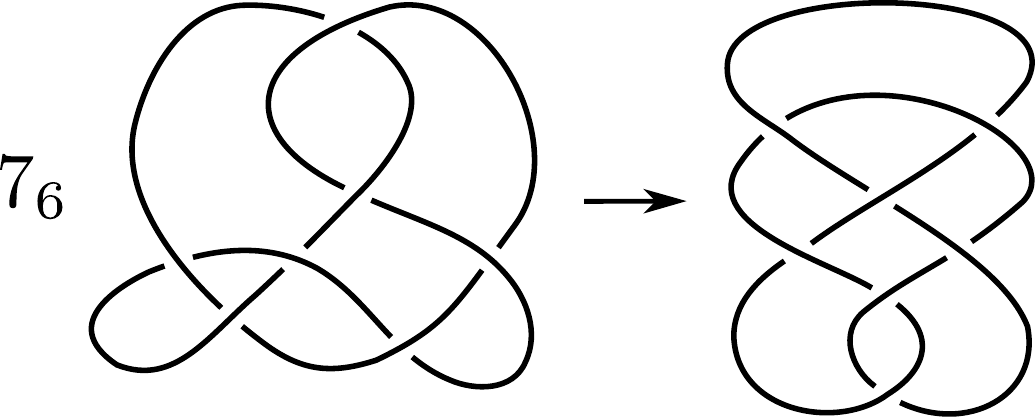}
\end{center}
\caption{Improvement of symmetry of $7_6$ in Rolfsen's knot table.\label{fig:rolfsen}}
\end{figure}
\begin{figure}[htbp]
\begin{center}
\includegraphics[scale=0.7]{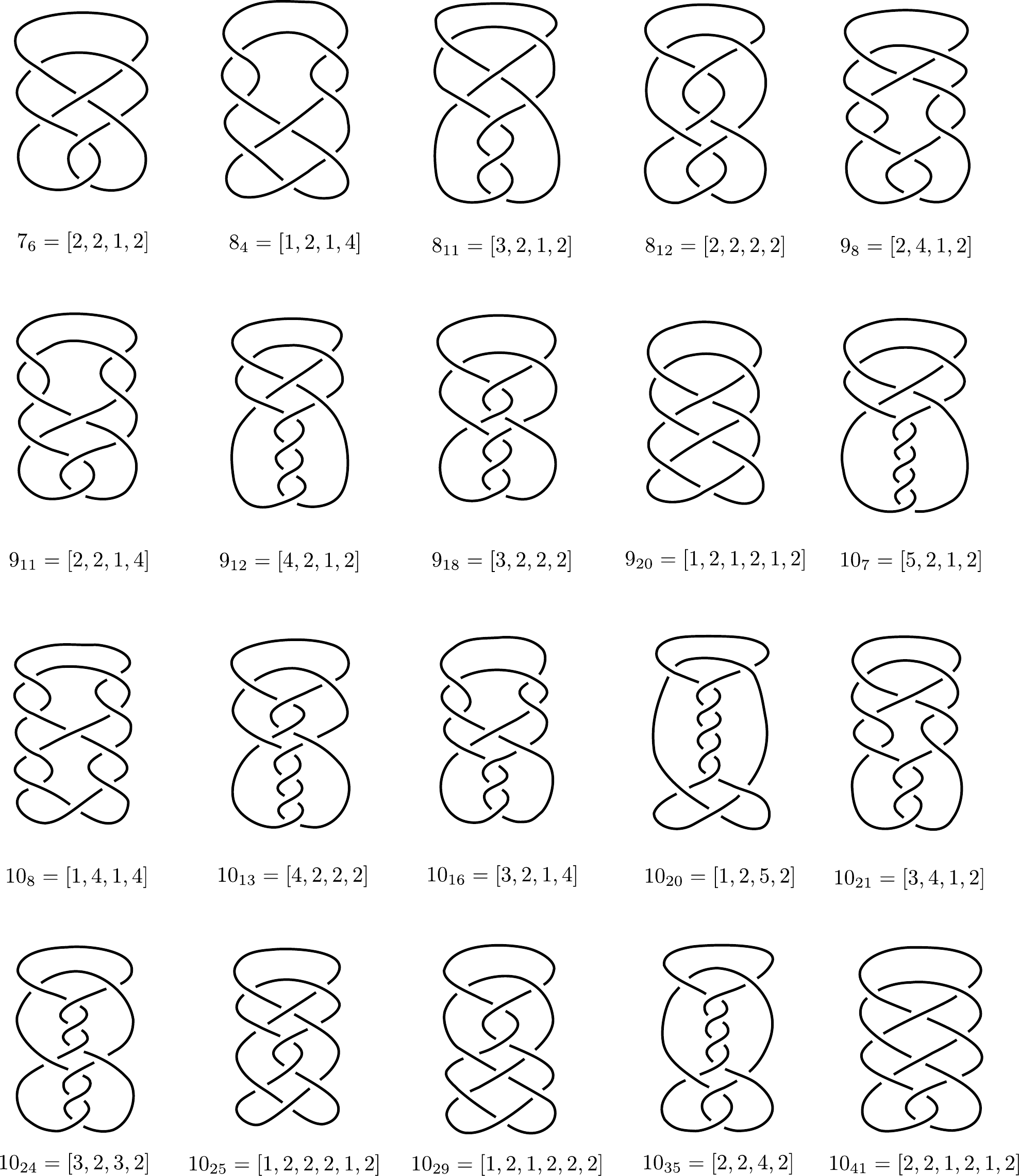}
\end{center}
\caption{Improvement of symmetry in Rolfsen's knot table.\label{fig:rolfsen_sym_table}}
\end{figure}
The value $c_2(K)$ tends to increase if $c(K)$ increases in Table \ref{tb:ct}, but it has not been proved yet. Unsolved questions regarding $c_2(K)$ and $c_t(K)$ are as follows.
\begin{qst}
\quad$(1)$ {\rm For any positive integer $n$, is there a knot $K$ such that $c_2(K)=c(K)+n$?}\\
\quad$(2)$ {\rm Is it possible to evaluate $c_2(K)$ in terms of $c(K)$? In particular, is the crossing number of semi-even expansion evaluated?}\\
\quad$(3)$ {\rm Can Lamm's conjecture (Conjecture $\ref{conj:lamm}$) be proved? If it is proved, $c_t(K)=c_2(K)$ for knots $K$ with $c_2(K)\leq c(K)+1$}.\\
\quad$(4)$ {\rm Can Lamm's conjecture be extended to the case $c_t(K)>c(K)$? If it is possible, $c_t(K)$ turns to be exactly the same as $c_2(K)$.}\\
\quad$(5)$ {\rm Is there any method to detect or evaluate $c_t(K)$ for general knots?}
\end{qst}
\section{Appendix}
\label{sec:appendix}
The algorithm to calculate $c_2(K)$ is described in this section. Let $K=K(p,q)$ $(p$ id odd, $q$ is even, $0<q<p$ and $\gcd(p,q)=1)$ be a two-bridge knot. $c_2(K)$ is determined by the following procedure. Figure \ref{fig:flow} shows the flowchart of the algorithm.\\
\begin{figure}[htbp]
\begin{center}
\includegraphics[scale=0.48]{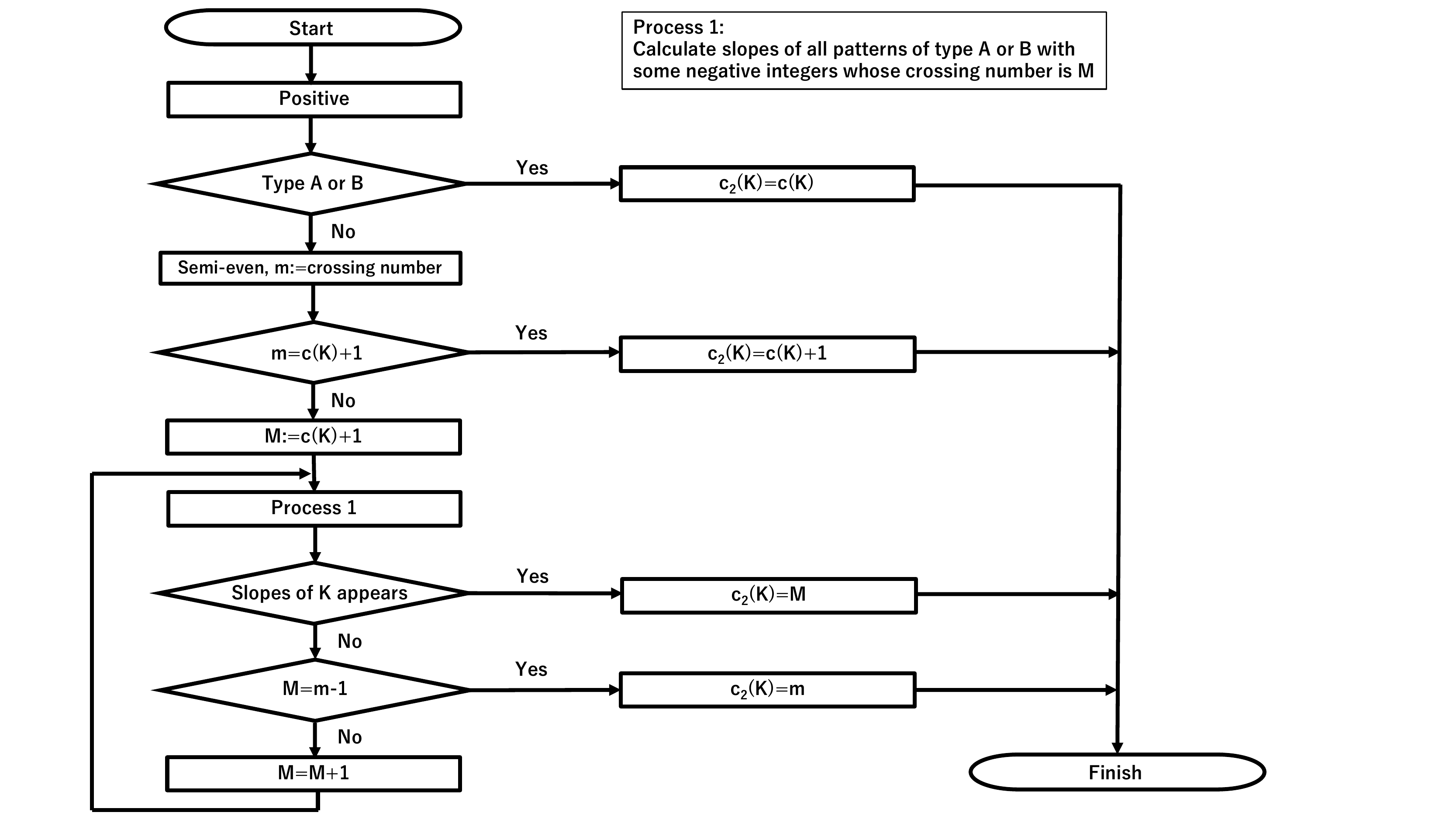}
\end{center}
\caption{Flowchart to calculate $c_2(K)$ for two-bridge knots.\label{fig:flow}}
\end{figure}
\quad {\it Step 1.}\quad Carry out the positive expansion of all four slopes of $K$. If one of them are either Type $A$ or $B$, then $c_2(K)=c(K)$. Otherwise $c_2(K)\geq c(K)+1$ and go on to Step\ $2$.\\
\quad {\it Step 2.}\quad Choose slopes of $K$ with even denominators and carry out the semi-even expansion. Let $m$ be the minimum crossing number of the expansions. Then $c_2(K)\leq m$. If $m=c(K)+1$, then $c_2(K)=c(K)+1$. If $m\geq c(K)+2$, then go on to Step\ $3$.\\
\quad {\it Step 3.}\quad List all expansions of Type $A$ or $B$ with positive integers whose sum is $c(K)+1$.\\
\quad {\it Step 4.}\quad Add some negative signs for each expansion and calculate the continued fraction for all patterns. If one of the four slopes of $K$ appears, $c_2(K)=c(K)+1$. Otherwise, go on to Step\ $5$.\\
\quad {\it Step 5.}\quad Add one more crossing and repeat from Step\ $3$. If $c_2(K)$ is not determined by $m-1$ crossings, $c_2(K)=m$.\\
\section*{Acknowledgement}
The author sincerely appreciates precious comments by Jun Murakami and Makoto Sakuma.

\end{document}